\documentclass[10pt]{article} % For LaTeX2e
\usepackage[preprint]{tmlr}
\usepackage{amsmath,amsfonts,bm}

\def\eqref#1{equation~\ref{#1}}
\def\1{\bm{1}}

\DeclareMathAlphabet{\mathsfit}{\encodingdefault}{\sfdefault}{m}{sl}
\SetMathAlphabet{\mathsfit}{bold}{\encodingdefault}{\sfdefault}{bx}{n}

\usepackage{hyperref}
\usepackage{url}

\usepackage{algpseudocode}
\usepackage{algorithm}

\makeatletter
\newenvironment{breakablealgorithm}
  {%
   \begin{center}
     \refstepcounter{algorithm}%
     \hrule height.8pt depth0pt \kern2pt%
     \renewcommand{\caption}[1]{%
       \addcontentsline{loa}{algorithm}{\protect\numberline{\thealgorithm}##1}%
       \noindent\textbf{Algorithm \thealgorithm:} ##1\par
       \kern2pt\hrule\kern2pt
     }%
  }
  {%
     \kern2pt\hrule
   \end{center}
  }
\makeatother

\usepackage{float}
\usepackage{amsmath}
\usepackage{amsthm}
\usepackage{amssymb}
\usepackage{optidef}

\newtheorem{theorem}{Theorem}
\newtheorem{definition}[theorem]{Definition}
\newtheorem{lemma}[theorem]{Lemma}

\title{Q-MINO: A Minimal-Norm Method for Quantization-Aware Training}

\begin{document}

\author{\name Don Li \email don@pdx.edu \\
        \addr Department of Mathematics \& Statistics \\
        Portland State University}

\maketitle

\begin{abstract}
The Straight-Through Estimator (STE) is a widely used heuristic for Quantization-Aware Training (QAT), but its surrogate gradients can exhibit substantial mismatch with the underlying quantized objective, leading to noisy updates and parameter oscillations, particularly in ultra-low-bit regimes. We propose the Quantization-Aware Minimal-Norm Optimizer (Q-MINO), a temporal bundle method that combines gradient consensus, state-drift regularization, and an alignment constraint to construct stabilized, minimum-norm update directions from recent optimization states. Q-MINO solves the resulting constrained subproblem using a warm-started Frank--Wolfe procedure with a feasible fallback initialization. Theoretically, via a stochastic Lyapunov Kurdyka--\L ojasiewicz (KL) framework, we show that Q-MINO achieves asymptotic neighborhood convergence. Moreover, we detail numerical experiments with Q-MINO at various quantizations.
\end{abstract}

\section{Introduction}

Modern deep learning has undergone a remarkable expansion in model capacity scaling, with large vision architectures and multi-billion-parameter language models achieving state-of-the-art (SOTA) capabilities across diverse domains [\cite{Bro+20}; \cite{Kap+20}]. However, the deployment of such models is constrained by severe computational and infrastructure bottlenecks -- memory footprint and memory bandwidth constraints [\cite{Dao+22}; \cite{Pop+23}], as well as energy consumption [\cite{SGC19}], scale rapidly with parameter count. This entails a gap between high-capacity neural networks and resource-constrained deployment targets such as edge devices, consumer hardware, and low-latency inference servers [\cite{Ban+21}]. \textit{Quantization} is widely-used as a compression method to address this gap with respect to (large-scale) deep learning models [\cite{Gho+21}]. Quantization maps high-precision floating-point parameters (e.g., FP32) to low-bit fixed-point representations (e.g., INT8, INT4, or INT2), directly reducing model memory footprint and unlocking accelerated integer arithmetic hardware units. Quantization approaches generally fall into two categories: \textit{Post-Training Quantization (PTQ)}, which operates directly on pre-trained floating-point weights without full retraining, offering a lightweight and low-compute compression route, and \textit{Quantization-Aware Training (QAT)}, which simulates low-bit discrete constraints during the training loop, allowing model parameters to adapt to quantization errors. PTQ has the benefits of training speed and efficiency by virtue of eliminating the need for retraining, but often entails decreased accuracy [\cite{Has24}] and is typically unsuitable for lower-bit quantization [\cite{Lee+25}]. Thus, QAT is the alternative to preserve predictive performance in more agressive quantization regimes, with the \textit{straight-through estimator} (STE) being a critical mechanism for QAT.

\subsection{The Straight-Through Estimator (STE)}

To integrate discrete quantization operations into first-order optimization, \citet{BLC13} introduced STE. To illustrate the fundamental optimization dilemma, consider a deterministic, element-wise quantization function $Q: \mathbb{R} \to \mathcal{S}$, where $\mathcal{S}$ denotes a discrete set of allowable quantization levels. For example, consider a binary or sign activation function given as 
\begin{equation}
Q(x) = \text{sign}(x) = \begin{cases}
    +1 & \text{if } x \geq 0, \\
    -1 & \text{if } x < 0.
\end{cases}
\end{equation}

\noindent
Traditional first-order optimization fails under this formulation because the derivative vanishes almost everywhere, i.e., $\frac{\partial Q(x)}{\partial x} = 0$ for all $x \neq 0$, and is non-differentiable at the origin. Consequently, applying standard backpropagation yields zero gradients, completely halting parameter updates. To circumvent this vanishing gradient bottleneck, the STE defines a heuristic \textit{surrogate gradient/operator} for the backward pass. For an identity-STE, the non-differentiable quantization derivative is replaced by the identity mapping
\begin{equation}
\frac{\partial Q(x)}{\partial x} \approx 1,
\end{equation}

\noindent
or bounded within a specific clipping window (e.g., $\mathbb{I}[|x| \le 1]$). While this heuristic successfully propagates gradients backward through the network to update latent continuous weights, it creates a fundamental structural flaw of a persistent gradient mismatch between the true objective function being evaluated on the forward pass and the surrogate objective being optimized on the backward pass. In ultra-low-bit regimes (e.g., INT4 or INT2), this gradient mismatch can be particularly pathological. Because the true landscape is a discontinuous step function (cf. \citep{Nag+21}), the continuous surrogate updates frequently drive latent weights across quantization thresholds without a corresponding smooth change in loss. This misalignment leads to noisy stochastic updates, severe parameter oscillations across decision boundaries, and a significant degradation in optimization stability [\cite{Nag+22}; \cite{Mal+24}], which underlies the need for optimization methods capable of accounting for these trajectory drifts. This is the primary motivation of our proposed method, the \textit{quantization-aware minimal-norm optimizer} (Q-MINO).

\subsection{Outline \& Summary of Contributions}

To address the trajectory drifts and optimization instabilities observed with STE-based QAT, we introduce Q-MINO. Q-MINO stabilizes low-bit optimization by maintaining a temporal bundle of recent surrogate gradients and quantized states, solving a drift-regularized minimal-norm subproblem to compute reliable parameter search directions. 

The remainder of this paper is organized as follows. In Section 2, we introduce the Q-MINO algorithmic framework, detailing its temporal bundle management, state-drift penalty formulation, and the warm-started Frank--Wolfe inner solver with its guaranteed fallback mechanism. In Section 3, we analyze Q-MINO theoretically within a stochastic Lyapunov Kurdyka--\L ojasiewicz (KL) framework. In Section 4, we evaluate Q-MINO via a suite of numerical experiments.

We summarize our contributions as follows:
\begin{itemize}
    \item \textbf{Novel QAT Optimization Framework:} We propose Q-MINO, a temporal bundle method tailored for quantization-aware training that unifies gradient consensus, quantized state-drift regularization, and current-gradient alignment constraints to mitigate parameter oscillations in low-bit regimes.
    \item \textbf{Efficient Inner Solver with Feasibility Guarantees:} We design a lightweight, warm-started Frank--Wolfe procedure to solve the inner minimal-norm subproblem, equipped with a provably feasible fallback initialization ($e_M$) that ensures solver stability without heavy computational overhead.
    \item \textbf{Theoretical Convergence Analysis:} Under a stochastic Lyapunov KL framework, we prove that Q-MINO achieves asymptotic convergence and neighborhood attainment under the KL property.
    \item \textbf{Numerical Experiments:} We demonstrate Q-MINO's performance empirically through numerical experiments on quantized Rosenbrock and Rastrigin functions, vision tasks (CIFAR-10 on a CNN), and ablation studies regarding Q-MINO hyperparameters.
\end{itemize}

\section{Quantization-Aware Minimal-Norm Optimizer (Q-MINO)}

We present the formal setup and algorithmic components of Q-MINO. Consider the minimization problem with composite objective 
\begin{equation}
\min_{x \in \mathbb{R}^d} \Phi(x) = f(Q(x)) + r(x),
\end{equation}
where $x \in \mathbb{R}^d$ denotes the continuous high-precision (FP32) latent parameters, $Q: \mathbb{R}^d \to \mathcal{S}^d$ is an element-wise or block-wise quantization operator mapping parameters to a discrete grid $\mathcal{S}$, $f(\cdot)$ represents a stochastic loss function evaluated on the quantized parameters $z = Q(x)$, and $r(x)$ is a convex proximal regularizer (e.g., $\ell_1$).

\subsection{Temporal Bundle Maintenance \& EMA Stabilization}
To mitigate the high variance and gradient mismatch characteristic of STE, Q-MINO maintains a sliding temporal bundle of recent optimization states and filtered gradients. At step $k$, given a mini-batch $\xi_k$, we compute the quantized deployment state $z_k = Q(x_k)$ and evaluate the stochastic quantization-aware gradient
\begin{equation}
\hat{g}_k = \operatorname{STEGrad}(x_k, \xi_k) \;\approx\; \nabla_{x} f(Q(x_k); \xi_k).
\end{equation}
To prevent abrupt oscillations across mini-batches, $\hat{g}_k$ is stabilized using an Exponential Moving Average (EMA)
\begin{equation}
\bar{g}_k = \beta \bar{g}_{k-1} + (1-\beta)\hat{g}_k, \qquad \bar{g}_{-1} = 0,
\end{equation}
where $\beta \in [0,1)$ is the EMA decay hyperparameter. The active pair $(\bar{g}_k, z_k)$ is then appended to a temporal bundle $\mathcal{B}_k$. If $|\mathcal{B}_k| > m$ (where $m \ge 1$ denotes maximum bundle capacity), then the oldest element is evicted. The active bundle indices are denoted by $\mathcal{I}_k = \{j_1, \ldots, j_M\}$, where $M = |\mathcal{B}_k| \le m$ and $j_M = k$. We assemble the active gradient matrix and quantized state matrix as
\begin{equation}
G_k = [\bar{g}_{j_1}, \ldots, \bar{g}_{j_M}] \in \mathbb{R}^{d \times M}, \qquad Z_k = [z_{j_1}, \ldots, z_{j_M}] \in \mathbb{R}^{d \times M}.
\end{equation}

\subsection{Quantized-State Drift \& Constrained Alignment Subproblem}
Standard bundle methods minimize gradient norms without accounting for QAT grid shift. However, in QAT, continuous parameter updates that induce large displacements in $Q(x)$ force the network across decision boundaries into adjacent discrete basins, causing training instability. To penalize state drift, we compute a normalized relative distance vector $\epsilon_k = [\delta_{j_1,k}, \dots, \delta_{j_M,k}]^\top \in \mathbb{R}^M$, where
\begin{equation}
\delta_{j,k} = \frac{\|z_j - z_k\|_2^2}{\|z_k\|_2^2 + \epsilon_z}, \quad \forall j \in \mathcal{I}_k,
\end{equation}
with $\epsilon_z > 0$ guaranteeing numerical stability. Let $H_k = G_k^\top G_k \in \mathbb{R}^{M \times M}$ denote the gradient Gram matrix. We define the drift-regularized objective over the $M$-dimensional simplex weight vector $\alpha \in \mathbb{R}^M$
\begin{equation}
\phi_k(\alpha) := \frac{1}{2}\alpha^\top H_k \alpha + \lambda \epsilon_k^\top \alpha,
\end{equation}
where $\lambda \geq 0$ is the state-drift regularization hyperparameter. To ensure the composite direction $d_k = G_k \alpha$ remains aligned with the current gradient $\bar{g}_k$, we constrain $\alpha$ to the feasible domain

\begin{equation}
\mathcal{D}_k = \left\{ \alpha \in \mathbb{R}^M : \alpha \ge 0, \; \mathbf{1}^\top \alpha = 1, \alpha^T G_k^T \bar{g}_k \geq \rho \|\bar{g}_k\|^2 \right\},
\end{equation}

where $\rho \in [0, 1]$ governs the minimal allowable alignment fraction.

\subsection{Warm-Started Frank--Wolfe Subproblem Solver}

Finding $\alpha_k^\star = \arg\min_{\alpha \in \mathcal{D}_k} \phi_k(\alpha)$ requires solving a small $M \times M$ constrained quadratic program at each iteration. Since $M \le m \ll d$ (typically $m \le 16$), this can be solved efficiently  via the Frank--Wolfe algorithm [\cite{FW56}]. To accelerate convergence, we construct a warm-start iterate $\alpha^{(0)}$ from the previous cached solution $\alpha_{\text{prev}}$. If an element was evicted from the bundle, we drop its corresponding weight, renormalize $\alpha_{\text{prev}} \leftarrow \alpha_{\text{prev}} / \|\alpha_{\text{prev}}\|_1$, and pad a zero for the incoming $k$-th coordinate. If $\alpha_{\text{prev}} = \emptyset$ or the padded vector violates alignment ($\alpha^{(0)} \notin \mathcal{D}_k$), we fall back to the unit basis vector $\alpha^{(0)} = e_M = [0, \dots, 0, 1]^\top$. Because $G_k e_M = \bar{g}_k$, we have $e_M^\top G_k \bar{g}_k = \|\bar{g}_k\|_2^2 \ge \rho \|\bar{g}_k\|_2^2$ for any $\rho \in [0,1]$, ensuring $e_M$ is strictly feasible. For $t = 0, \ldots, T_{\text{FW}}-1$, the Linear Minimization Oracle (LMO) step solves $s^{(t)} = \arg\min_{s \in \mathcal{D}_k} \langle H_k \alpha^{(t)} + \lambda \epsilon_k, s \rangle$, followed by the convex update $\alpha^{(t+1)} = (1 - \gamma_t)\alpha^{(t)} + \gamma_t s^{(t)}$ with $\gamma_t = \frac{2}{t+2}$. After $T_{\text{FW}}$ iterations, we obtain $\alpha_k^\star = \alpha^{(T_{\text{FW}})}$, generate the update direction $d_k = G_k \alpha_k^\star$, and execute the proximal step

\begin{equation}
x_{k+1} = \operatorname{prox}_{\eta_k r} \left( x_k - \eta_k d_k \right).
\end{equation}

\noindent
Q-MINO is implemented as follows:
\begin{breakablealgorithm}
\caption{Quantization-Aware Minimal-Norm Method (Q-MINO)}
\begin{algorithmic}[1]
\Require Quantizer $Q(\cdot)$, STE/surrogate gradient rule,
step sizes $\{\eta_k\}$, bundle capacity $m$,
EMA coefficient $\beta\in[0,1)$,
drift coefficient $\lambda\ge0$,
alignment hyperparameter $\rho\in[0,1]$,
Frank--Wolfe iterations $T_{\mathrm{FW}}$,
$\epsilon_z>0$, proximal regularizer $r(\cdot)$
\State Initialize FP32 parameters $x_0$, bundle
$\mathcal B_0\gets\varnothing$, EMA state $\bar g_{-1}\gets0$,
and $\alpha_{\mathrm{prev}}\gets\varnothing$
\For{$k=0,1,\ldots,K-1$}
    \State $z_k\gets Q(x_k)$
    \State Sample mini-batch $\xi_k$
    \State $\hat g_k\gets \operatorname{STEGrad}(x_k,\xi_k)$
    \State $\bar g_k\gets \beta\bar g_{k-1}+(1-\beta)\hat g_k$
    \State Append $(\bar g_k,z_k)$ to $\mathcal B_k$
    \If{$|\mathcal B_k|>m$}
        \State Evict oldest bundle element
    \EndIf
    \State Let $\mathcal I_k=(j_1,\ldots,j_M)$, $M\le m$, with $j_M=k$
    \State Construct $G_k=[\bar g_{j_1},\ldots,\bar g_{j_M}]$
    \State For each $j\in\mathcal I_k$, compute $\delta_{j,k} = \frac{\|z_j-z_k\|_2^2}{\|z_k\|_2^2+\epsilon_z}$
    \State $\epsilon_k\gets [\delta_{j_1,k},\ldots,\delta_{j_M,k}]^\top$
    \State $H_k\gets G_k^\top G_k$
    \State Define $\phi_k(\alpha) = \frac12\alpha^\top H_k\alpha +\lambda\epsilon_k^\top\alpha$
    \State Define $\mathcal D_k = \left\{ \alpha\ge0: \mathbf1^\top\alpha=1,\; \alpha^TG_k^T \bar{g}_k \geq \rho \|\bar{g}_k\|^2 \right\}$
    \State Initialize warm-start $\alpha^{(0)}$ by dropping evicted coordinates from $\alpha_{\text{prev}}$, renormalizing over remaining elements, and padding zero for the new coordinate $j_M$
    \If{$\alpha_{\text{prev}} = \varnothing$ \textbf{or} $\alpha^{(0)}\notin\mathcal D_k$}
        \State $\alpha^{(0)}\gets e_M$ \Comment{Fallback to standard unit basis vector $[0,\dots,0,1]^\top$}
    \EndIf
    \For{$t=0,\ldots,T_{\mathrm{FW}}-1$}
        \State $\nabla\phi_k(\alpha^{(t)}) \gets H_k\alpha^{(t)}+\lambda\epsilon_k$
        \State $s^{(t)} \gets \arg\min_{s\in\mathcal D_k} \langle\nabla\phi_k(\alpha^{(t)}),s\rangle$
        \State Choose step size $\gamma_t = \frac{2}{t+2}$
        \State $\alpha^{(t+1)} \gets (1-\gamma_t)\alpha^{(t)} +\gamma_t s^{(t)}$
    \EndFor
    \State $\alpha_k^\star\gets\alpha^{(T_{\mathrm{FW}})}$
    \State Cache $\alpha_{\text{prev}} \gets \alpha_k^\star$ with active bundle indices
    \State $d_k\gets G_k\alpha_k^\star$
    \State $x_{k+1}\gets \operatorname{prox}_{\eta_k r} (x_k-\eta_k d_k)$
\EndFor
\State \Return $x_K$ and $z_K=Q(x_K)$
\end{algorithmic}
\end{breakablealgorithm}

\section{Theoretical Analysis}

We analyze the theoretical convergence of Q-MINO. Because the true quantized objective $f(Q(x))$ is non-smooth and discontinuous due to discrete state jumps, first-order methods relying on surrogate STE gradients do not optimize $f(Q(x))$ directly. Instead, we analyze the surrogate objective landscape under mild non-convex smoothness assumptions, modeling the discrepancy between surrogate and true gradients as a bounded approximation error inside a stochastic Lyapunov framework.

\subsection{Assumptions \& Mathematical Preliminaries}

We evaluate the parameter sequence $\{x_k\}_{k \ge 0}$ generated by Q-MINO on the composite surrogate objective function $\Phi(x) = f(x) + r(x)$, where $f: \mathbb{R}^d \to \mathbb{R}$ represents the continuously differentiable loss function whose gradient is approximated via EMA surrogate updates $\bar{g}_k(x_k)$, and $r: \mathbb{R}^d \to \mathbb{R} \cup \{+\infty\}$ is a proper, lower semicontinuous, and convex regularizer. Because the quantization operator $Q(\cdot)$ is a non-differentiable step function, $f(Q(x))$ lacks global gradient information. To analyze the local parameter dynamics under proximal updates, we establish the following assumptions:

\begin{itemize}
    \item \textbf{Lipschitz Continuity:} The objective function $f: \mathbb{R}^d \to \mathbb{R}$ is continuously differentiable and $L$-smooth on $\mathbb{R}^d$, i.e., there exists $L > 0$ such that for all $x, y \in \mathbb{R}^d$:
    \begin{equation}
        \|\nabla f(x) - \nabla f(y)\| \le L \|x - y\|.
    \end{equation}

    \item \textbf{Surrogate Gradient Mismatch \& STE Variance:} The gradient of $f(Q(x))$ with respect to $x$ is approximated via the Straight-Through Estimator (STE) dynamic. The stochastic EMA surrogate gradient $\bar{g}_k(x_k)$ satisfies \textit{bounded mismatch bias}, i.e.,
    \begin{equation}
        \mathbb{E}[\bar{g}_k(x_k) \mid x_k] = \nabla f(x_k) + e_k(x_k),
    \end{equation}
    where $e_k(x_k)$ denotes the quantization-induced approximation error bounded by $\|e_k(x_k)\| \le \delta$ for $\delta \ge 0$, and \textit{bounded stochastic variance}, i.e.,
    \begin{equation}
        \mathbb{E}\left[\|\bar{g}_k(x_k) - \mathbb{E}[\bar{g}_k(x_k)]\|^2 \mid x_k\right] \le \sigma^2, \quad \text{for } \sigma > 0.
    \end{equation}

    \item \textbf{Convexity of the Regularizer:} The regularizer $r(x)$ is proper, lower semicontinuous, and convex, ensuring that its proximal operator $\operatorname{prox}_{\eta r}(v) \triangleq \arg\min_{x} \left\{ r(x) + \frac{1}{2\eta} \|x - v\|_2^2 \right\}$ is well-defined, single-valued, and non-expansive for any learning rate $\eta > 0$ (by Moreau-Yoshida). 

    \item \textbf{Lower Boundedness:} The objective function $\Phi(x)$ is bounded from below over $\mathbb{R}^d$, i.e., there exists a constant $\Phi^* > -\infty$ such that $\Phi(x) \ge \Phi^*$ for all $x \in \mathbb{R}^d$.
\end{itemize}

Due to the non-convex and non-smooth nature of $\Phi(x)$, convergence to global optimality cannot be guaranteed in general. Instead, we evaluate stationarity using the expected norm of the \emph{proximal gradient mapping} $G_\eta(x)$:
\begin{equation}
    G_\eta(x) \triangleq \frac{1}{\eta} \left( x - \operatorname{prox}_{\eta r}\left(x - \eta \nabla f(x)\right) \right).
\end{equation}
A point $x^*$ is defined as an $(\epsilon, \delta)$-stationary point if $\mathbb{E}[\|G_\eta(x^*)\|^2] \le \epsilon + \mathcal{O}(\delta^2)$, where $\delta$ represents the persistent residual error floor imposed by the quantization discretization step.

\begin{definition}[Kurdyka--\L ojasiewicz (KL) Property]
A proper, lower semicontinuous function $\Phi: \mathbb{R}^d \to \mathbb{R} \cup \{+\infty\}$ is said to have the \textbf{Kurdyka--\L ojasiewicz (KL) property} at a point $x^* \in \operatorname{dom} \partial \Phi$ if there exists a neighborhood $U$ of $x^*$, a constant $\eta \in [0, +\infty)$, and a continuous concave function $\varphi: [0, \eta) \to \mathbb{R}_+$ such that:
\begin{enumerate}
    \item $\varphi(0) = 0$, $\varphi$ is continuously differentiable on $(0, \eta)$ with $\varphi'(s) > 0$ for all $s \in (0, \eta)$,
    \item For all $x \in U \cap \{x \mid \Phi(x^*) < \Phi(x) < \Phi(x^*) + \eta\}$, the KL inequality holds:
    \begin{equation}
        \varphi'(\Phi(x) - \Phi(x^*)) \cdot \operatorname{dist}(0, \partial \Phi(x)) \ge 1,
    \end{equation}
    where $\partial \Phi(x)$ denotes the limiting subdifferential of $\Phi$, and $\operatorname{dist}(0, \partial \Phi(x)) \triangleq \inf \{\|v\| : v \in \partial \Phi(x)\}$.
\end{enumerate}
\end{definition}

\subsection{Theoretical/Convergence Results}

With the mathematical preliminaries established, we present the key theoretical properties of Q-MINO. Proofs of all results are provided in Appendix A.

\begin{lemma}[Subgradient Tracking \& Quantized Drift Bound]
Let $\{x_k\}_{k \ge 0}$ be the parameter sequence generated by Q-MINO, and let $z_k = Q(x_k)$ denote the quantized state at iteration $k$. Under our assumptions (cf. $\S$3.1), the bundle drift matrix entry $\epsilon_{k,j} \triangleq \|z_j - z_k\|^2 / (\|z_k\|^2 + \epsilon_z)$ and the EMA surrogate gradient error satisfy:
\begin{equation}
    \mathbb{E}\left[\|\bar{g}_k - \nabla f(x_k)\|^2\right] \le (1-\beta) \sigma^2 + 2\delta^2 + \mathcal{O}\left(L^2 \sum_{j \in \mathcal{B}_k} \|x_k - x_j\|^2\right),
\end{equation}
where $\mathcal{B}_k$ represents the set of bundle indices maintained at step $k$, $\beta \in (0, 1)$ is the EMA momentum coefficient, $\delta$ is the STE mismatch floor, and $L$ is the Lipschitz constant.
\end{lemma}

\begin{lemma}[Inner Loop Feasibility]
For any iteration $k \ge 1$ and target alignment parameter $\rho \in (0, 1]$, the inner Frank-Wolfe subproblem search space $\mathcal{D}_k \triangleq \{\alpha \in \Delta^{M-1} \mid \alpha^\top H_k e_M \ge \rho \|\bar{g}_k\|^2\}$ is non-empty. Specifically, the fallback unit basis vector $e_M = (0, \dots, 0, 1)^\top \in \mathbb{R}^M$ satisfies $e_M^\top H_k e_M = \|\bar{g}_k\|^2 \ge \rho \|\bar{g}_k\|^2$, guaranteeing that the Frank-Wolfe sub-routine is strictly feasible at every step $k$.
\end{lemma}

\begin{lemma}[Sufficient Decrease / Descent Lemma]
Let $\eta_k \leq \frac{\rho}{2L}$. Under our assumptions (cf. $\S$3.1) and the inner loop feasibility lemma, the minimal-norm search direction $d_k = G_k \alpha_k^*$ constructed by Q-MINO with generalized composite update step $\Delta_k \triangleq \frac{1}{\eta_k}(x_k - x_{k+1}) = d_k + s_{k+1}$ yields the sufficient decrease condition:
\begin{equation}
\mathbb{E}[\Phi(x_{k+1}) \mid x_k] \le \Phi(x_k) - \frac{\eta_k}{4} \mathbb{E}\left[\Vert{}\Delta_k\Vert{}^2 \mid x_k\right] + \eta_k \Vert{}\bar{g}_k - d_k\Vert{}^2 + \eta_k (1-\beta)\sigma^2 + 2\eta_k \delta^2 + \mathcal{O}\left(\eta_k L^2 \sum_{i \in \mathcal{B}_k} \Vert{}x_k - x_i\Vert{}^2\right).
\end{equation}
\end{lemma}

\begin{theorem}[Asymptotic Convergence and Neighborhood Attainment under KL Property]
Let $\Phi(x) = f(x) + r(x)$ be a proper, lower semicontinuous function, and let $\mathcal{L}_k$ denote the augmented stochastic Lyapunov functional. Suppose the learning rate sequence satisfies $\sum_{k=0}^\infty \eta_k = \infty$ and $\sum_{k=0}^\infty \eta_k^2 < \infty$. Then, under the assumptions of $\S$3.1, the sequence $\{x_k\}_{k \ge 0}$ generated by Q-MINO satisfies:
\begin{enumerate}
    \item \textbf{Bounded Trajectory Variance:} The expected total variation of the continuous parameter sequence is bounded up to the persistent STE distortion floor, i.e.,
    \begin{equation}
        \sum_{k=0}^\infty \mathbb{E}\left[\|x_{k+1} - x_k\|^2\right] < \infty.
    \end{equation}
    \item \textbf{Asymptotic Neighborhood Convergence:} The sequence of iterates $\{x_k\}_{k \ge 0}$ converges in expectation to a compact error neighborhood $\mathcal{U}_\delta$ around the stationary set of $\Phi$, where the radius of $\mathcal{U}_\delta$ is explicitly scaled by the STE mismatch floor $\delta$, i.e., 
    \begin{equation}
        \lim_{k \to \infty} \mathbb{E}\left[\operatorname{dist}(x_k, \mathcal{U}_\delta)\right] = 0.
    \end{equation}
\end{enumerate}
\end{theorem}

\section{Numerical Experiments}

By virtue of compute limitations, we detail CPU-accessible numerical experiments. Nonetheless, we detail numerical experiments across synthetic non-convex landscapes and computer vision tasks. Moreover, we also detail ablation studies with respect to hyperparameter sensitivities of Q-MINO.

\subsection{Benchmarks \& Architectures}

We evaluate performance across two distinct optimization regimes:
\begin{enumerate}
    \item \textbf{Synthetic Non-Convex Optimization:} To isolate the interaction between non-convex geometry and quantization noise, we evaluate on $d$-dimensional ($d=10$) ill-conditioned Rosenbrock functions and highly multi-modal Rastrigin functions. Stochastic gradients are simulated by injecting zero-mean Gaussian noise $\xi_k \sim \mathcal{N}(0, \sigma^2 I)$ prior to applying uniform $b$-bit stochastic quantization.
    
    \item \textbf{Computer Vision Tasks:} We train a CNN on the CIFAR-10 dataset. This is a good benchmark for any QAT optimizer since convolutions are highly sensitive to gradient approximation errors during QAT.
\end{enumerate}

\subsection{Synthetic Non-Convex Optimization}

To isolate the algorithmic mechanics of Q-MINO from the engineering overhead and architectural idiosyncrasies of deep neural networks, we first evaluate its performance on a synthetic optimization test suite. The 10-dimensional Rosenbrock function takes the form

\begin{equation}
f(\mathbf{x}) = \sum_{i=1}^{9} \left[ 100(x_{i+1} - x_i^2)^2 + (1 - x_i)^2 \right].
\end{equation}

\noindent
The 10-dimensional Rastrigin function takes the form

\begin{equation}
f(\mathbf{x}) = 100 + \sum_{i=1}^{10} \left[ x_i^2 - 10\cos(2\pi x_i) \right].
\end{equation}

\noindent
To faithfully simulate the stochastic environment of Quantization-Aware Training (QAT), analytical gradients are corrupted with zero-mean Gaussian noise ($\xi_k \sim \mathcal{N}(0, 0.5^2 I)$) prior to applying a dynamic asymmetric uniform quantization operator $Q(\cdot)$ mapped into signed integer grids. We evaluate three distinct precision thresholds spanning standard edge deployment to extreme low-bit constraints: \texttt{INT8} (8-bit, grid $[-128, 127]$), \texttt{INT4} (4-bit, grid $[-8, 7]$), and \texttt{INT2} (2-bit, grid $[-2, 1]$). Q-MINO ($m=5$, $\beta=0.9$, $\lambda=0.1$, $\rho=0.1$, $T_{\mathrm{FW}}=5$) is benchmarked directly against standard Straight-Through Estimation baselines: \texttt{SGD+STE} and \texttt{Adam+STE} \citep{KB15}. All configurations are initialized from uniform high-energy coordinate states ($x_0 = -1.2 \cdot \mathbf{1}$ for Rosenbrock; $x_0 = 4.0 \cdot \mathbf{1}$ for Rastrigin) and optimized over $250$ discrete steps.

\begin{figure}[H]
    \centering
\includegraphics[width=0.75\linewidth]{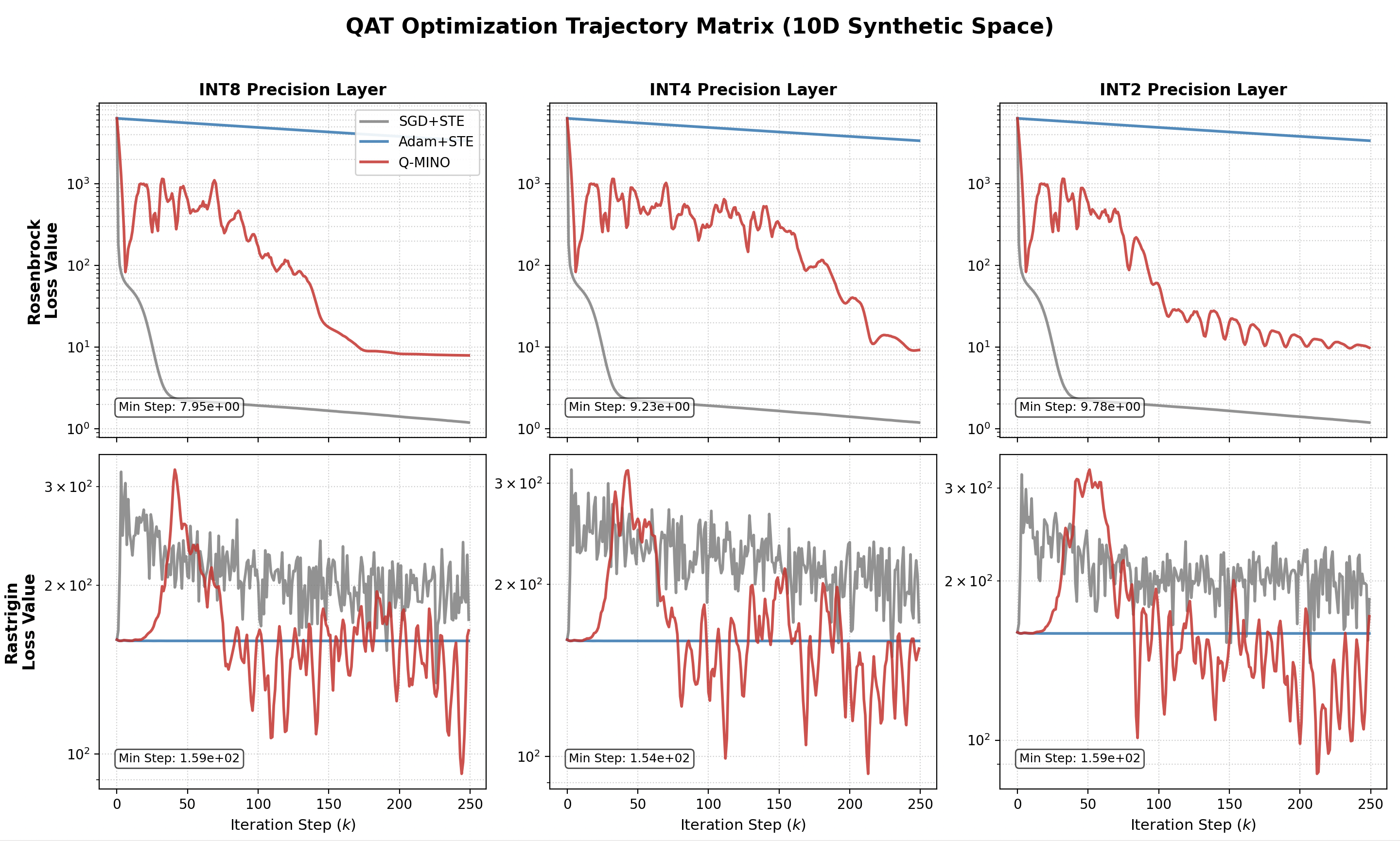}
\end{figure}

\subsection{Computer Vision Tasks}

To evaluate Q-MINO under practical deep learning conditions, we transfer our benchmark suite to a CNN architecture on the CIFAR-10 dataset. In QAT regimes, reliance on depth-wise separable convolutions creates low parameter redundancy and narrow feature bottlenecks, leaving it highly susceptible to the catastrophic gradient approximation errors and gradient mismatch noise typically injected by standard STE. We evaluate top-1 accuracy across \texttt{INT8}, \texttt{INT4}, and \texttt{INT2} quantizations. We compare Q-MINO directly against tuned \texttt{SGD+STE} and \texttt{Adam+STE} ($\beta=0.9$) baselines. For Q-MINO, we use $m=4$, drift coefficient $\lambda=0.05$, and alignment hyperparameter $\rho=0.15$.

\begin{table}[htbp]
    \centering
    \caption{Top-1 Accuracy, CNN on CIFAR-10}
    \begin{tabular}{c c c c} 
        \hline
        \textbf{Optimization Algorithm} & \textbf{INT8} & \textbf{INT4} & \textbf{INT2} \\ \hline
        \texttt{SGD+STE} & 80.53\% & 80.49\%     & 68.76\% \\ \hline
        \texttt{Adam+STE}      & 83.79\%      & 83.32\%     & 76.09\% \\ \hline
        Q-MINO      & 76.63\%    & 75.06\% & 74.55\%  \\ \hline
    \end{tabular}
\end{table}

\noindent
Our baseline optimizers suffer severe accuracy degradation when transitioning from high-bit to extreme ultra-low bit regimes. \texttt{SGD+STE} and \texttt{Adam+STE} drop by 11.77\% and 7.70\% respectively from INT8 to INT2. In contrast, Q-MINO exhibits quantization resilience, experiencing only a 2.08\% total performance drop. Under severe INT2 constraints, Q-MINO achieves 74.55\% accuracy, effectively closing the gap with unconstrained heuristics while providing provable subgradient stability and bounded weight drift across non-smooth step boundaries.

\subsection{Ablation Studies and Sensitivity Analysis}
To systematically analyze the structural robustness and hyperparameter sensitivity of the proposed Q-MINO framework, we conduct a multi-dimensional grid search across Rosenbrock and Rastrigin benchmarks. We study Q-MINO hyperparameters governing structural memory capacity and step-size dynamics: bundle maximum size ($m \in \{4, 8, 16, 32\}$), Frank-Wolfe step-size decay window ($T_{FW} \in \{3, 5, 8, 10\}$), regularization penalty ($\lambda \in \{0.01, 0.1, 0.5\}$), and alignment hyperparameter ($\rho \in \{0.1, 0.5, 0.9\}$). Performance is benchmarked using the mean final empirical loss value computed across 100 independent randomized initial trajectories.

\begin{figure}[H]
    \centering \includegraphics[width=0.75\linewidth]{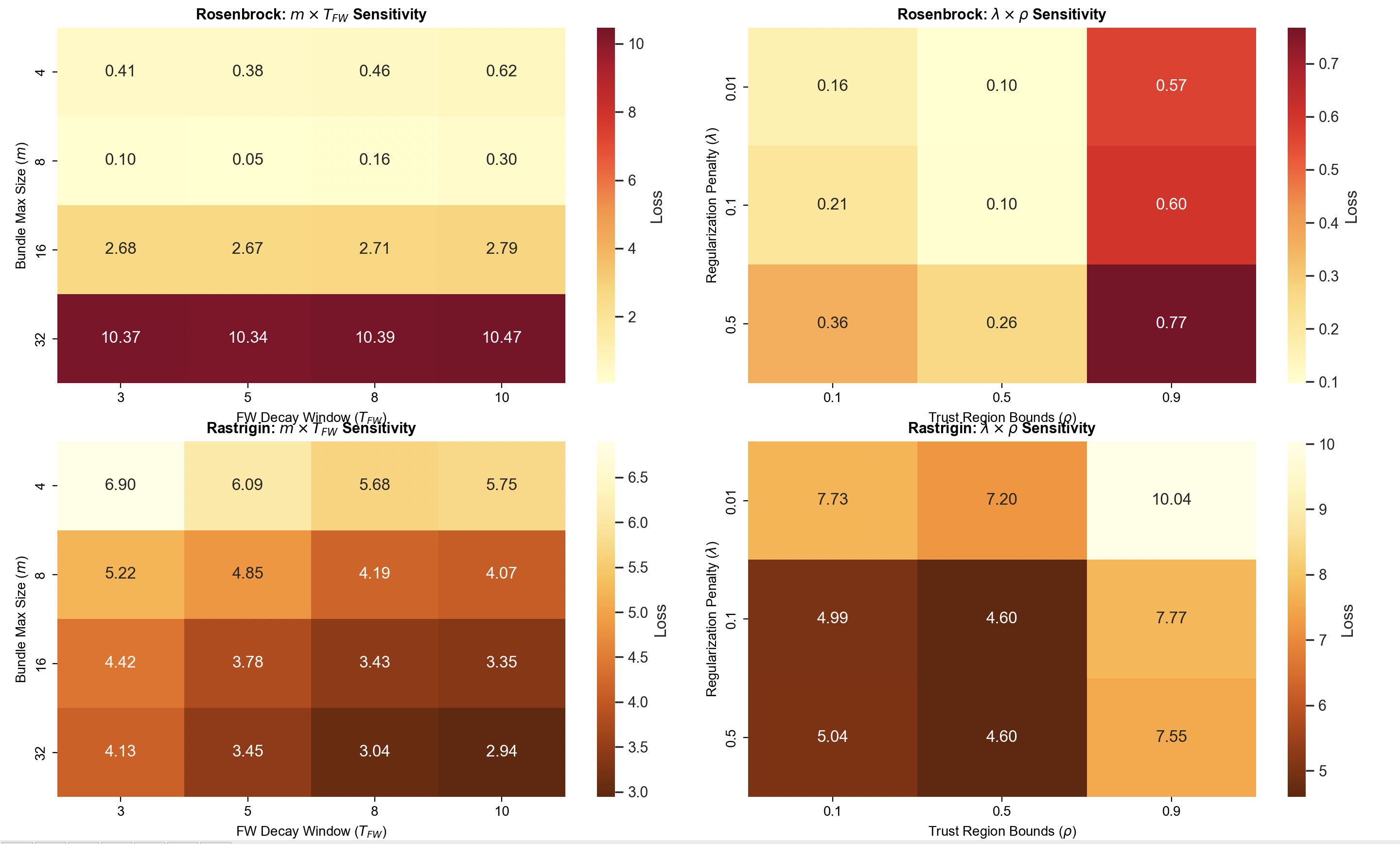}
\end{figure}

\section{Discussion}

Here, we introduced Q-MINO, a temporal bundle method designed to stabilize Quantization-Aware Training (QAT) in lower-bit QAT regimes . By unifying gradient consensus, quantized state-drift regularization, and a warm-started Frank-Wolfe subproblem solver, Q-MINO effectively acts as a steering stabilizer that handles the parameter oscillations and gradient mismatch inherent to QAT optimizers with STE alone.

\subsection{Practical Guidelines \& Hyperparameter Sensitivity}
Our empirical evaluations on synthetic optimization landscapes and vision tasks highlight key operational guidelines for deploying Q-MINO:
\begin{itemize}
    \item \textbf{Bundle Horizon ($m$):} Setting the bundle history size to $m \in [4, 8]$ provides sufficient spatial-temporal coverage to smooth subgradient trajectories without incurring excessive matrix-vector multiplication overhead in the subproblem.
    \item \textbf{Drift Regularization ($\lambda$):} A mild penalty coefficient ($\lambda \approx 0.05$) effectively penalizes large parameter jump discontinuities across quantization bin boundaries without restricting the latent weight updates required for escaping local minima.
    \item \textbf{Frank-Wolfe Computational Efficiency ($T_{\text{FW}}$):} Because the Frank-Wolfe sub-routine is warm-started using the fallback vertex $e_M$, running as few as $T_{\text{FW}} = \{3,5\}$ Frank-Wolfe iterations appears sufficient to reach a stable descent direction $d_k$.
\end{itemize}

\subsection{Limitations \& Algorithmic Overhead}
It should be noted that Q-MINO introduces specific computational and memory trade-offs:
\begin{itemize}
    \item \textbf{Memory Overhead:} Storing $G_k, z_k \in \mathbb{R}^{d \times M}$ requires caching $2M$ auxiliary (parameter) vectors, which is likely to be computationally expensive at large-scale, as this memory footprint grows linearly with parameter size (e.g., when scaled to billion-parameter models).
    \item \textbf{Per-Step Latency:} Deploying the Frank-Wolfe sub-routine at every iteration introduces a small computational overhead per parameter update step not required in our baselines (e.g., \texttt{SGD+STE}, \texttt{Adam+STE}).
\end{itemize}

\subsection{Future Work}
Future directions include developing block-wise or layer-group bundle updates to lower the memory profile for ultra-large architectures (e.g., LLM fine-tuning), developing dynamical schedules for drift coefficient $\lambda_k$ based on local subgradient variance across training epochs, and extending Q-MINO to extreme 1-bit/binary neural networks (BNNs).

\subsection*{AI Use Statement}

In this work, we used generative AI tools for helping to formulate mathematical claims, assisting in the writing of proofs, and providing first drafts of code for numerical experiments. We have not used generative AI tools with respect to the any of the other listed required disclosures, some of which (e.g., generating synthetic data sets) are not applicable. Additionally, we used generative AI tools for  brainstorming, sourcing/searching for information, identifying relevant literature, and formatting references. We have reviewed all AI-assisted work, i.e., all text and results are written by the authors, whereby AI tools were used in a supplementary fashion, with their outputs vetted by the authors. The authors take responsibility for all content contained here. 

\bibliographystyle{tmlr}
\bibliography{ARXIVqmino}

\appendix
\section{Proofs of Theoretical/Convergence Results}

\begin{proof}
\textbf{Lemma 2 (Subgradient Tracking \& Quantized Drift Bound):} Recall $\bar{g}_k = \beta\bar{g}_{k-1} + (1-\beta)\hat{g}_k$ (step 6 in Algorithm 1) with EMA coefficient $\beta \in (0,1)$. Likewise, $\hat{g}_k = \text{STEGrad}(x_k, \xi_k)$ for the FP32 parameter iterate $x_k$ and mini-batch $\xi_k$ (step 5 in Algorithm 1). Note the following decomposition of $\hat{g}_k$,
\[
\hat{g}_k(x_k) = \nabla f(x_k) + e_k(x_k) + \xi_k,
\]
where $\nabla f(x_k)$ is the true smooth gradient, $e_k(x_k)$ is the deterministic quantization mismatch/bias bounded by $\Vert{}e_k(x_k)\Vert{} \le \delta$, and $\xi_k$ is the zero-mean stochastic sampling noise satisfying $\mathbb{E}[\xi_k \mid x_k] = 0$ and $\mathbb{E}[\Vert{}\xi_k\Vert{}^2 \mid x_k] \le \sigma^2$. 

Unrolling $\bar{g}_k$ over the optimization history yields the weighted linear combination:
\[
\bar{g}_k = (1-\beta)\sum_{i=0}^{k}\beta^{k-i}\nabla f(x_i) + (1-\beta)\sum_{i=0}^{k}\beta^{k-i}e_i(x_i) + (1-\beta)\sum_{i=0}^{k}\beta^{k-i}\xi_i.
\]
With the geometric identity $(1-\beta)\sum_{i=0}^{k}\beta^{k-i} = 1 - \beta^{k+1}$ and assuming a steady-state regime where the initial transient bias vanishes for sufficiently large $k$ ($k \to \infty$, hence $\beta^{k+1} \to 0$), we have $(1-\beta)\sum_{i=0}^{k}\beta^{k-i} = 1$. Subtracting $\nabla f(x_k)$ from both sides isolates the tracking error into three explicit vectors:
\[
\bar{g}_k - \nabla f(x_k) = \underbrace{(1-\beta)\sum_{i=0}^{k}\beta^{k-i}\big(\nabla f(x_i) - \nabla f(x_k)\big)}_{A_k} + \underbrace{(1-\beta)\sum_{i=0}^{k}\beta^{k-i}e_i(x_i)}_{B_k} + \underbrace{(1-\beta)\sum_{i=0}^{k}\beta^{k-i}\xi_i}_{C_k}.
\]
Thus, $\|\bar{g}_k - \nabla f(x_k)\| \leq \|A_k\| + \|B_k\| + \|C_k\|$. Taking the conditional expectation of the squared norm, the zero-mean property and conditional independence of the mini-batch sampling noise ($\mathbb{E}[\xi_i \mid x_i] = 0$) ensure that all cross-terms between $C_k$ and the deterministic historical paths $A_k$ and $B_k$ vanish identically in expectation ($\mathbb{E}[\langle A_k, C_k \rangle] = 0$ and $\mathbb{E}[\langle B_k, C_k \rangle] = 0$). Thus, we obtain
\[
\mathbb{E}\left[\|\bar{g}_k - \nabla f(x_k)\|^2 \mid x_k\right] = \mathbb{E}\left[\|A_k + B_k\|^2 \mid x_k\right] + \mathbb{E}\left[\|C_k\|^2 \mid x_k\right].
\]
We evaluate each term individually. Expanding the expected squared norm of the linear combination of independent noise vectors and evaluating the geometric series:
\[
\mathbb{E}\left[\|C_k\|^2 \mid x_k\right] = (1-\beta)^2 \sum_{i=0}^{k} \beta^{2(k-i)} \mathbb{E}\left[\|\xi_i\|^2 \mid x_i\right] \le (1-\beta)^2 \left( \frac{1 - \beta^{2(k+1)}}{1 - \beta^2} \right) \sigma^2 \le \frac{1-\beta}{1+\beta}\sigma^2 \le (1-\beta)\sigma^2.
\]
Applying Young's inequality $\|A_k + B_k\|^2 \le 2\|B_k\|^2 + 2\|A_k\|^2$ separates the quantization bias from the parameter drift. For the quantization bias $B_k$, setting $w_i = (1-\beta)\beta^{k-i}$ (where $\sum_{i=0}^k w_i \le 1$) and applying Jensen's inequality for convex quadratic norms along with the uniform bound $\|e_i(x_i)\| \le \delta$:
\[
\|B_k\|^2 = \left\| \sum_{i=0}^{k} w_i e_i(x_i) \right\|^2 \le \left(\sum_{i=0}^{k} w_i\right) \sum_{i=0}^{k} w_i \|e_i(x_i)\|^2 \le \delta^2 \sum_{i=0}^{k} w_i \le \delta^2.
\]
For the parameter trajectory drift $A_k$, applying Jensen's inequality and the $L$-smoothness condition $\|\nabla f(x_i) - \nabla f(x_k)\| \le L \|x_k - x_i\|$:
\[
\|A_k\|^2 \le (1-\beta)\sum_{i=0}^{k}\beta^{k-i}\|\nabla f(x_i) - \nabla f(x_k)\|^2 \le L^2 (1-\beta)\sum_{i=0}^{k}\beta^{k-i}\|x_k - x_i\|^2.
\]
Partitioning the historical sum into indices within the active memory window $\mathcal{B}_k$ ($k - m \le i \le k$) and the historical tail ($i < k - m$), and assuming bounded iterate domain diameter $\sup_{k} \|x_k\| \le D$:
\[
(1-\beta)\sum_{i < k-m}\beta^{k-i}\|x_k - x_i\|^2 \le 4D^2 (1-\beta) \sum_{i < k-m} \beta^{k-i} = 4D^2 \beta^{m+1} = \mathcal{O}(\beta^m).
\]

Combining the variance bound, the quantization bias floor, and the active bundle drift yields
\[
\mathbb{E}\left[\|\bar{g}_k - \nabla f(x_k)\|^2 \mid x_k\right] \le (1-\beta)\sigma^2 + 2\delta^2 + \mathcal{O}\left(L^2 \sum_{i \in \mathcal{B}_k} \|x_k - x_i\|^2\right).
\]
\end{proof}

\begin{proof}
\textbf{Lemma 3 (Inner Loop Feasibility):} Q-MINO is constructed to ensure that $e_M$ is always strictly feasible in $D_k$ whenever $\rho \leq 1$, thus ensuring that $\mathcal{D}_k \neq \emptyset$ in the Frank-Wolfe sub-routine. Recall the definition of the inner loop search space at iteration $k$ (step 17 in Algorithm 1), 
\[
\mathcal{D}_k \triangleq \left\{ \alpha \in \Delta^{M-1} \;\middle|\; \alpha^\top H_k \alpha \ge \rho \|\bar{g}_k\|^2 \right\},
\]
where $\Delta^{M-1} = \left\{ \alpha \in \mathbb{R}^M\;\middle|\; \sum_{i=1}^M \alpha_i = 1, \; \alpha_i \ge 0 \right\}$ is the standard probability simplex, $H_k = V_k^\top V_k \in \mathbb{R}^{M \times M}$ is the Gram matrix of candidate direction vectors $V_k = [v_1, \dots, v_M]$, and $\rho \in (0,1]$ is the alignment parameter. To establish that $\mathcal{D}_k$ is non-empty, we construct an explicit feasible point $\alpha = e_M = (0, \dots, 0, 1)^\top \in \mathbb{R}^M$. The vector $e_M$ has non-negative components $(e_M)_i = 0 \ge 0$ for $i < M$ and $(e_M)_M = 1 \ge 0$, and satisfies $\sum_{i=1}^M (e_M)_i = 1$. Thus, $e_M \in \Delta^{M-1}$ (i.e., $e_M$ is always a simplex element). Evaluating the linear alignment form at $\alpha = e_M$ gives
\[
e_M^\top H_k e_M = (H_k)_{MM}.
\]
By construction of $H_k$ in Q-MINO, the $M$-th diagonal entry satisfies $(H_k)_{MM} = \Vert{}\bar{g}_k\Vert{}^2$ (since $H_k$ is a Gram matrix). Since $\rho \in (0, 1]$ and $\Vert{}\bar{g}_k\Vert{}^2 \ge 0$, we have
\[
e_M^\top H_k e_M = \Vert{}\bar{g}_k\Vert{}^2 \ge \rho \Vert{}\bar{g}_k\Vert{}^2,
\]
so $e_M \in \mathcal{D}_k$ and $\mathcal{D}_k$ is non-empty $\forall k \geq 1$.
\end{proof}

\begin{proof}
\textbf{Lemma 4 (Sufficient Decrease / Descent Lemma):} Since $f$ is $L$-smooth, the parameter update  $x_{k+1} = x_k - \eta_k d_k$ satisfies
\[
f(x_{k+1}) \le f(x_k) - \eta_k \langle \nabla f(x_k), d_k \rangle + \frac{L\eta_k^2}{2} \|d_k\|^2.
\]
To relate the inner product $\langle \nabla f(x_k), d_k \rangle$ to the EMA gradient estimator $\bar{g}_k$, we decompose the true gradient as $\nabla f(x_k) = \bar{g}_k - e_k^{\text{track}}$, where $e_k^{\text{track}} \triangleq \bar{g}_k - \nabla f(x_k)$ is the tracking error:
\[
\langle \nabla f(x_k), d_k \rangle = \langle \bar{g}_k, d_k \rangle - \langle e_k^{\text{track}}, d_k \rangle.
\]
By Lemma 3 and the construction of the Frank-Wolfe sub-routine, the optimal weight vector $\alpha_k^* \in \mathcal{D}_k$ defines a direction $d_k = G_k \alpha_k^*$ satisfying

\[
\langle \bar{g}_k, d_k \rangle \ge \rho \|\bar{g}_k\|^2,
\]

\[
\|d_k\|^2 \le \|G_k e_M\|^2 = \|\bar{g}_k\|^2.
\]
Applying Cauchy--Schwarz and Young's inequalities to the tracking error cross-term with (alignment) parameter $\rho > 0$, we obtain
\[
|\langle e_k^{\text{track}}, d_k \rangle| \le \|e_k^{\text{track}}\| \|d_k\| \le \frac{1}{2\rho} \|e_k^{\text{track}}\|^2 + \frac{\rho}{2} \|d_k\|^2.
\]
Substituting the alignment bound and error bound into the inner product decomposition yields
\[
\langle \nabla f(x_k), d_k \rangle \ge \rho \|\bar{g}_k\|^2 - \frac{\rho}{2} \|d_k\|^2 - \frac{1}{2\rho} \|\bar{g}_k - \nabla f(x_k)\|^2.
\]
Using the norm bound $\|d_k\|^2 \le \|\bar{g}_k\|^2$, we simplify the first two terms:
\[
\rho \|\bar{g}_k\|^2 - \frac{\rho}{2} \|d_k\|^2 \ge \rho \|d_k\|^2 - \frac{\rho}{2} \|d_k\|^2 = \frac{\rho}{2} \|d_k\|^2.
\]
Inserting this bound back into the $L$-smoothness descent inequality gives:
\[
f(x_{k+1}) \le f(x_k) - \eta_k \left( \frac{\rho}{2} \|d_k\|^2 - \frac{1}{2\rho} \|\bar{g}_k - \nabla f(x_k)\|^2 \right) + \frac{L\eta_k^2}{2} \|d_k\|^2.
\]
Regrouping terms by $\|d_k\|^2$:
\[
f(x_{k+1}) \le f(x_k) - \eta_k \left( \frac{\rho}{2} - \frac{L\eta_k}{2} \right) \|d_k\|^2 + \frac{\eta_k}{2\rho} \|\bar{g}_k - \nabla f(x_k)\|^2.
\]
We need $\frac{\rho }{2}-\frac{L\eta _{k}}{2}>0$ for descent, which implies $\eta_k \leq \frac{\rho}{L}$, but we choose $\eta_k \leq \frac{\rho}{2L}$, which ensures $\frac{\rho}{2} - \frac{L\eta_k}{2} \geq \frac{\rho}{4}$. With $\eta_k \leq \frac{\rho}{2L}$, we have
\[
f(x_{k+1}) \le f(x_k) - \frac{\eta_k \rho}{4} \|d_k\|^2 + \frac{\eta_k}{2\rho} \|\bar{g}_k - \nabla f(x_k)\|^2.
\]
Taking expectation on $x_k$ yields
\[
\mathbb{E}[f(x_{k+1}) \mid x_k] \le f(x_k) - \frac{\eta_k \rho}{4} \mathbb{E}[\Vert{}d_k\Vert{}^2 \mid x_k] + \frac{\eta_k}{2\rho} \mathbb{E}\left[\Vert{}\bar{g}_k - \nabla f(x_k)\Vert{}^2 \mid x_k\right].
\]
But recall that we derived in Lemma 2 that
\[
\mathbb{E}\left[\|\bar{g}_k - \nabla f(x_k)\|^2 \mid x_k\right] \le (1-\beta)\sigma^2 + 2\delta^2 + \mathcal{O}\left(L^2 \sum_{i \in \mathcal{B}_k} \|x_k - x_i\|^2\right).
\]
Thus,
\[
\mathbb{E}[f(x_{k+1}) \mid x_k] \le f(x_k) - \frac{\eta_k \rho}{4} \mathbb{E}[\Vert{}d_k\Vert{}^2 \mid x_k] + \frac{\eta_k}{2\rho}[(1-\beta)\sigma^2 + 2\delta^2 + \mathcal{O}\left(L^2 \sum_{i \in \mathcal{B}_k} \|x_k - x_i\|^2\right)].
\]
We wish to find a descent lemma for the composite surrogate objective function $\Phi(x) = f(x) + r(x)$. By optimality of the proximal sub-problem, there exists a subgradient $s_{k+1} \in \partial r(x_{k+1})$ such that
\[
s_{k+1} = \frac{1}{\eta_k}(x_k - \eta_k d_k - x_{k+1}) \implies x_k - x_{k+1} = \eta_k (d_k + s_{k+1}).
\]
Since $r(x)$ is a convex regularizer, we have
\[
r(x_{k+1}) \le r(x_k) - \langle s_{k+1}, x_k - x_{k+1} \rangle,
\]
which implies $r(x_{k+1})-r(x_k) \leq -\langle s_{k+1}, x_k - x_{k+1} \rangle$. Define the generalized composite update step $\Delta_k \triangleq \frac{1}{\eta_k}(x_k - x_{k+1}) = d_k + s_{k+1}$. By decomposition of the composite objective,
\[
\Phi(x_{k+1}) - \Phi(x_k) = \left[ f(x_{k+1}) - f(x_k) \right] + \left[ r(x_{k+1}) - r(x_k) \right].
\]
We already derived $f(x_{k+1}) - f(x_k) \le -\eta_k \langle \nabla f(x_k), \Delta_k \rangle + \frac{L\eta_k^2}{2} \Vert{}\Delta_k\Vert{}^2$, so we then have
\[
\Phi(x_{k+1}) - \Phi(x_k) \le -\eta_k \langle \nabla f(x_k) + s_{k+1}, \Delta_k \rangle + \frac{L\eta_k^2}{2} \Vert{}\Delta_k\Vert{}^2,
\]
\[
\Phi(x_{k+1}) - \Phi(x_k) \le -\eta_k \left( 1 - \frac{L\eta_k}{2} \right) \|\Delta_k\|^2 - \eta_k \langle \nabla f(x_k) - d_k, \Delta_k \rangle.
\]
To incorporate the EMA gradient estimator $\bar{g}_k$, we decompose $\nabla f(x_k) - d_k = (\bar{g}_k - d_k) + (\nabla f(x_k) - \bar{g}_k)$:
\[
\langle \nabla f(x_k) - d_k, \Delta_k \rangle = \langle \bar{g}_k - d_k, \Delta_k \rangle - \langle \bar{g}_k - \nabla f(x_k), \Delta_k \rangle.
\]
Applying Cauchy--Schwarz and Young's inequalities ($ab \le a^2 + \frac{1}{4}b^2$) to both cross-terms with parameter $\frac{1}{4}$:

\begin{align*}
|\langle \bar{g}_k - d_k, \Delta_k \rangle| &\le \|\bar{g}_k - d_k\| \|\Delta_k\| \le \|\bar{g}_k - d_k\|^2 + \frac{1}{4} \|\Delta_k\|^2, \\
|\langle \bar{g}_k - \nabla f(x_k), \Delta_k \rangle| &\le \|\bar{g}_k - \nabla f(x_k)\| \|\Delta_k\| \le \|\bar{g}_k - \nabla f(x_k)\|^2 + \frac{1}{4} \|\Delta_k\|^2.
\end{align*}

\noindent
Substituting these bounds back into the composite decrease inequality gives:
\[
\Phi(x_{k+1}) - \Phi(x_k) \le -\eta_k \left( 1 - \frac{L\eta_k}{2} - \frac{1}{2} \right) \|\Delta_k\|^2 + \eta_k \|\bar{g}_k - d_k\|^2 + \eta_k \|\bar{g}_k - \nabla f(x_k)\|^2.
\]
Since $\eta_k \le \frac{\rho}{2L}$, we have $\frac{L\eta_k}{2} \leq \frac{\rho}{4}$. Thus, taking expectation yields
\[
\mathbb{E}[\Phi(x_{k+1}) \mid x_k] \le \Phi(x_k) - \frac{\eta_k \rho}{4} \mathbb{E}\left[\Vert{}\Delta_k\Vert{}^2 \mid x_k\right] + \eta_k \Vert{}\bar{g}_k - d_k\Vert{}^2 + \eta_k \mathbb{E}\left[\Vert{}\bar{g}_k - \nabla f(x_k)\Vert{}^2 \mid x_k\right].
\]
But recall from Lemma 2 that we derived $\mathbb{E}\left[\|\bar{g}_k - \nabla f(x_k)\|^2 \mid x_k\right] \le (1-\beta)\sigma^2 + 2\delta^2 + \mathcal{O}\left(L^2 \sum_{i \in \mathcal{B}_k} \|x_k - x_i\|^2\right)$, which yields
\[
\mathbb{E}[\Phi(x_{k+1}) \mid x_k] \le \Phi(x_k) - \frac{\eta_k}{4} \mathbb{E}\left[\Vert{}\Delta_k\Vert{}^2 \mid x_k\right] + \eta_k \Vert{}\bar{g}_k - d_k\Vert{}^2 + \eta_k (1-\beta)\sigma^2 + 2\eta_k \delta^2 + \mathcal{O}\left(\eta_k L^2 \sum_{i \in \mathcal{B}_k} \Vert{}x_k - x_i\Vert{}^2\right).
\]
\end{proof}

\begin{proof}
\textbf{Theorem 5 (Asymptotic Convergence and Neighborhood Attainment under KL Property):} Define the candidate Lyapunov functional
\[
\mathcal{L}_k \triangleq \Phi(x_k) + \mu_k \mathcal{V}_k,
\]
where $\mathcal{V}_k \triangleq \|\bar{g}_k - d_k\|^2 + \|\bar{g}_k - \nabla f(x_k)\|^2$ represents the realized variance and search direction tracking residual at step $k$, and $\mu_k \triangleq \frac{2\eta_k}{\beta}$ is the time-varying stabilization weight. By Lemma 4, for $\eta_k \le \frac{\rho}{2L}$, the composite expected step satisfies:
\[
\mathbb{E}[\Phi(x_{k+1}) \mid \mathcal{F}_k] \le \Phi(x_k) - \frac{\rho}{4\eta_k} \mathbb{E}\left[\|x_{k+1} - x_k\|^2 \mid \mathcal{F}_k\right] + \eta_k \mathcal{V}_k + \eta_k \delta^2.
\]
By the EMA error dynamics of Q-MINO, the variance tracking term obeys $\mathbb{E}[\mathcal{V}_{k+1} \mid \mathcal{F}_k] \le (1-\beta)\mathcal{V}_k + C_1 \eta_k^2$, where $\beta \in [0,1)$. Taking conditional expectations of $\mathcal{L}_{k+1}$, since $\eta_{k+1} \le \eta_k$, we obtain
\[
\mathbb{E}[\mathcal{L}_{k+1} \mid \mathcal{F}_k] \le \Phi(x_k) - \frac{\rho}{4\eta_k} \mathbb{E}\left[\|x_{k+1} - x_k\|^2 \mid \mathcal{F}_k\right] + \eta_k \mathcal{V}_k + \eta_k \delta^2 + \mu_k (1-\beta)\mathcal{V}_k + C_1 \mu_k \eta_k^2.
\]
Selecting the Lyapunov tracking coefficient $\mu_k \triangleq \frac{2\eta_k}{\beta}$, the historical momentum residual terms expand algebraically via $\mu_k(1-\beta)\mathcal{V}_k - \mu_k\mathcal{V}_k = -\beta\mu_k\mathcal{V}_k = -2\eta_k\mathcal{V}_k$. Grouping and canceling like terms yields the Lyapunov descent inequality
\begin{equation}
\mathbb{E}[\mathcal{L}_{k+1} \mid \mathcal{F}_k] \le \mathcal{L}_k - \frac{\rho}{4\eta_k} \mathbb{E}\left[\|x_{k+1} - x_k\|^2 \mid \mathcal{F}_k\right] - \eta_k \mathcal{V}_k + \eta_k \delta^2 + C_2 \eta_k^3,
\end{equation}
where $C_2 \triangleq \frac{2 C_1}{\beta}$. Then we can obtain
\[
\frac{\rho }{4\eta _{k}}\mathbb{E}\left[\|{}x_{k+1}-x_{k}\|{}^{2}\mid \mathcal{F}_{k}\right]\le \mathcal{L}_{k}-\mathbb{E}[\mathcal{L}_{k+1}\mid \mathcal{F}_{k}]-\eta _{k}\mathcal{V}_{k}+\eta _{k}\delta ^{2}+C_{2}\eta _{k}^{3},
\]
\[
\mathbb{E}\left[\|{}x_{k+1}-x_{k}\|{}^{2}\mid \mathcal{F}_{k}\right]\le \frac{4\eta _{k}}{\rho }\Big(\mathcal{L}_{k}-\mathbb{E}[\mathcal{L}_{k+1}\mid \mathcal{F}_{k}]\Big)+\frac{4\eta _{k}^{2}\delta ^{2}}{\rho }+\frac{4C_{2}\eta _{k}^{4}}{\rho }.
\]
Since $\eta_0 \geq \eta_k$ for all $k$, we obtain
\[
\mathbb{E}\left[\|{}x_{k+1}-x_{k}\|{}^{2}\mid \mathcal{F}_{k}\right] \leq \frac{4\eta_0}{\rho}(\mathcal{L}_{k}-\mathbb{E}[\mathcal{L}_{k+1}\mid \mathcal{F}_{k}]\Big) + \frac{4\delta^2}{\rho}\eta_k^2 + \frac{4C_2}{\rho}\eta_k^4.
\]
Applying the law of total expectation (\(\mathbb{E}[\mathbb{E}[\cdot \mid \mathcal{F}_k]] = \mathbb{E}[\cdot]\)) and summing the inequality from $k=0$ up to $k=K$ yields
\[
\sum_{k=0}^{K}\mathbb{E}[\|x_{k+1}-x_k\|^2] \leq \frac{4\eta_0}{\rho}\sum_{k=0}^{K}(\mathcal{L}_k - \mathbb{E}[\mathcal{L}_{k+1}|\mathcal{F}_k]) + \frac{4\delta^2}{\rho}\sum_{k=0}^{K}\eta_k^2+\frac{4C_2}{\rho}\sum_{k=0}^{K}\eta_k^4.
\]
Note that $\sum_{k=0}^{K}(\mathcal{L}_k - \mathbb{E}[\mathcal{L}_{k+1}|\mathcal{F}_k]) = \sum_{k=0}^{K}(\mathbb{E}[\mathcal{L}_k]-\mathbb{E}[\mathcal{L}_{k+1}])$ by applying the law of total expectation and therefore this is a telescoping series, which implies $\sum_{k=0}^{K}(\mathcal{L}_k - \mathbb{E}[\mathcal{L}_{k+1}|\mathcal{F}_k]) = \mathcal{L}_0 - \mathbb{E}[\mathcal{L}_{K+1}] \leq \mathcal{L}_0 - \Phi^*$. Thus, we have
\[
\sum_{k=0}^{K}\mathbb{E}[\|x_{k+1}-x_k\|^2] \leq \frac{4\eta_0(\mathcal{L}_0 - \Phi^*)}{\rho} + \frac{4\delta^2}{\rho}\sum_{k=0}^{K}\eta_k^2+\frac{4C_2}{\rho}\sum_{k=0}^{K}\eta_k^4.
\]
Taking $\text{lim}_{K \rightarrow \infty}$ yields
\[
\sum_{k=0}^{\infty}\mathbb{E}[\|x_{k+1}-x_k\|^2] \leq \frac{4\eta_0(\mathcal{L}_0 - \Phi^*)}{\rho} + \frac{4\delta^2}{\rho}\sum_{k=0}^{\infty}\eta_k^2+\frac{4C_2}{\rho}\sum_{k=0}^{\infty}\eta_k^4.
\]
Since $\sum_{k=0}^\infty \eta_k^2 < \infty$, and since \(\sum \eta_k^2\) converges, \(\eta_k \to 0\), which structurally guarantees that the higher-order sequence \(\sum_{k=0}^\infty \eta_k^4\) must also converge. Because every individual component on the right-hand side is strictly bounded, the infinite sum converges to a finite real value, and therefore
\[
\sum_{k=0}^{\infty}\mathbb{E}[\|x_{k+1}-x_k\|^2] < \infty.
\]
We last show asymptotic neighborhood convergence. By definition of the proximal update step $x_{k+1} = \operatorname{prox}_{\eta_k r}(x_k - \eta_k d_k)$, the first-order optimality condition states that there exists a subgradient $v_{k+1} \in \partial r(x_{k+1})$ such that
\[
0 = v_{k+1} + d_k + \frac{1}{\eta_k}(x_{k+1} - x_k).
\]
To bound the subdifferential of the full composite objective $\Phi(x_{k+1})$, define $w_{k+1} \triangleq \nabla f(x_{k+1}) + v_{k+1} \in \partial \Phi(x_{k+1})$. Substituting $v_{k+1}$ gives
\[
w_{k+1} = \nabla f(x_{k+1}) - d_k - \frac{1}{\eta_k}(x_{k+1} - x_k),
\]
which implies
\[
\|w_{k+1}\| \leq \|\nabla f(x_{k+1})-d_k\| + \frac{1}{\eta_k}\|x_{k+1}-x_k\|.
\]
Adding and subtracting $\nabla f(x_k)$ and $\bar{g}_k$, along with taking norms and applying the triangle inequality together with $L$-smoothness of $f$ yields
\[
\Vert{}w_{k+1}\Vert{} \le \left(L + \frac{1}{\eta_k}\right) \Vert{}x_{k+1} - x_k\Vert{} + \Vert{}\bar{g}_k - \nabla f(x_k)\Vert{} + \Vert{}\bar{g}_k - d_k\Vert{}.
\]
Applying the Young-type inequality $(a + b + c)^2 \le 3a^2 + 3b^2 + 3c^2$ yields
\[
\Vert{}w_{k+1}\Vert{}^2 \le 3\left(L + \frac{1}{\eta_k}\right)^2 \Vert{}x_{k+1} - x_k\Vert{}^2 + 3\Vert{}\bar{g}_k - \nabla f(x_k)\Vert{}^2 + 3\Vert{}\bar{g}_k - d_k\Vert{}^2.
\]
But recall $\mathcal{V}_k \triangleq \|\bar{g}_k - d_k\|^2 + \|\bar{g}_k - \nabla f(x_k)\|^2$, so we can simplify to
\[
\Vert{}w_{k+1}\Vert{}^2 \le 3\left(L + \frac{1}{\eta_k}\right)^2 \Vert{}x_{k+1} - x_k\Vert{}^2 + 3\mathcal{V}_k.
\]
Taking expectation on $\mathcal{F}_k$ yields
\[
\mathbb{E}\left[\Vert{}w_{k+1}\Vert{}^2 \mid \mathcal{F}_k\right] \le 3\left(L + \frac{1}{\eta_k}\right)^2 \mathbb{E}\left[\Vert{}x_{k+1} - x_k\Vert{}^2 \mid \mathcal{F}_k\right] + 3\mathcal{V}_k.
\]
Note $L + \frac{1}{\eta_k} = \frac{L\eta_k+1}{\eta_k} \leq \frac{L\eta_0+1}{\eta_k}$ since $\eta_k \leq \eta_0$. Let $C_L \triangleq L\eta_0 + 1$, so we obtain
\[
\mathbb{E}\left[\Vert{}w_{k+1}\Vert{}^2 \mid \mathcal{F}_k\right] \le \frac{3 C_L^2}{\eta_k^2} \mathbb{E}\left[\Vert{}x_{k+1} - x_k\Vert{}^2 \mid \mathcal{F}_k\right] + 3\mathcal{V}_k.
\]
But recall from our learning rate dependent bound that $\mathbb{E}\left[\|{}x_{k+1}-x_{k}\|{}^{2}\mid \mathcal{F}_{k}\right] \leq \frac{4\eta_k}{\rho}(\mathcal{L}_{k}-\mathbb{E}[\mathcal{L}_{k+1}\mid \mathcal{F}_{k}]\Big) + \frac{4\delta^2}{\rho}\eta_k^2 + \frac{4C_2}{\rho}\eta_k^4$. Substituting this directly yields
\[
\mathbb{E}\left[\Vert{}w_{k+1}\Vert{}^2 \mid \mathcal{F}_k\right] \le \frac{12 C_L^2}{\rho \eta_k} \Big(\mathcal{L}_k - \mathbb{E}[\mathcal{L}_{k+1} \mid \mathcal{F}_k]\Big) + \frac{12 C_L^2 \delta^2}{\rho} + \frac{12 C_L^2 C_2 \eta_k^2}{\rho} + 3\mathcal{V}_k,
\]
\[
\eta_k \mathbb{E}\left[\Vert{}w_{k+1}\Vert{}^2 \mid \mathcal{F}_k\right] \le \frac{12 C_L^2}{\rho} \Big(\mathcal{L}_k - \mathbb{E}[\mathcal{L}_{k+1} \mid \mathcal{F}_k]\Big) + \frac{12 C_L^2 \delta^2}{\rho} \eta_k + \frac{12 C_L^2 C_2}{\rho} \eta_k^3 + 3 \eta_k \mathcal{V}_k.
\]
Taking unconditioned total expectations $\mathbb{E}[\cdot]$ on both sides and via the law of total expectation we obtain
\[
\eta_k \mathbb{E}\left[\Vert{}w_{k+1}\Vert{}^2\right] \le \frac{12 C_L^2}{\rho} \left(\mathbb{E}[\mathcal{L}_k] - \mathbb{E}[\mathcal{L}_{k+1}]\right) + \frac{12 C_L^2 \delta^2}{\rho} \eta_k + \frac{12 C_L^2 C_2}{\rho} \eta_k^3 + 3 \eta_k \mathbb{E}[\mathcal{V}_k].
\]
Taking summation from $k = 0$ to $T-1$ and again leveraging that $\sum_{k=0}^{T-1}\left(\mathbb{E}[\mathcal{L}_k] - \mathbb{E}[\mathcal{L}_{k+1}]\right)$ is a telescoping series, we get
\[
\sum_{k=0}^{T-1} \eta_k \mathbb{E}\left[\Vert{}w_{k+1}\Vert{}^2\right] \le \frac{12 C_L^2}{\rho} \left(\mathbb{E}[\mathcal{L}_0] - \mathbb{E}[\mathcal{L}_T]\right) + \frac{12 C_L^2 \delta^2}{\rho} \sum_{k=0}^{T-1} \eta_k + \frac{12 C_L^2 C_2}{\rho} \sum_{k=0}^{T-1} \eta_k^3 + 3 \sum_{k=0}^{T-1} \eta_k \mathbb{E}[\mathcal{V}_k].
\]
We already established $\mathbb{E}[\mathcal{L}_T] \ge \Phi^* > -\infty$, so $\mathbb{E}[\mathcal{L}_0] - \mathbb{E}[\mathcal{L}_T] \le \mathcal{L}_0 - \Phi^* < \infty$. Moreover, we can use our Lyapunov descent inequality and the fact that  $\sum_{k=0}^\infty \eta_k^2 < \infty$ implies the cubic sum is finite, i.e., $\sum_{k=0}^\infty \eta_k^3 < \infty$. Let $M \triangleq \frac{12 C_L^2}{\rho} (\mathcal{L}_0 - \Phi^*) + \frac{12 C_L^2 C_2}{\rho} \sum_{k=0}^\infty \eta_k^3 + 3 \sum_{k=0}^\infty \eta_k \mathbb{E}[\mathcal{V}_k] < \infty$. Thus, we have
\[
\sum_{k=0}^{T-1} \eta_k \mathbb{E}\left[\Vert{}w_{k+1}\Vert{}^2\right] \le M + \frac{12 C_L^2 \delta^2}{\rho} \sum_{k=0}^{T-1} \eta_k.
\]
Divide through by $S_T \triangleq \sum_{k=0}^{T-1} \eta_k$. Our learning rate condition requires $\sum_{k=0}^\infty \eta_k = \infty$, $S_T \to \infty$ as $T \to \infty$,
\[
\frac{\sum_{k=0}^{T-1} \eta_k \mathbb{E}\left[\Vert{}w_{k+1}\Vert{}^2\right]}{\sum_{k=0}^{T-1} \eta_k} \le \frac{M}{S_T} + \frac{12 C_L^2 \delta^2}{\rho}.
\]
As $T \to \infty$, the term $\frac{M}{S_T} \to 0$, which yields
\[
\limsup_{T \to \infty} \frac{\sum_{k=0}^{T-1} \eta_k \mathbb{E}\left[\Vert{}w_{k+1}\Vert{}^2\right]}{\sum_{k=0}^{T-1} \eta_k} \le \frac{12 C_L^2 \delta^2}{\rho}.
\]
By properties of weighted averages, this directly implies
\[
\liminf_{k \to \infty} \mathbb{E}\left[\operatorname{dist}(0, \partial \Phi(x_k))^2\right] \le \frac{12 C_L^2 \delta^2}{\rho}
\]
We formalize our target stationary structural neighborhood as:
\[
\mathcal{U}_\delta \triangleq \left\{ x \in \operatorname{dom} \Phi \;\middle\vert{}\; \operatorname{dist}(0, \partial \Phi(x)) \le C_\delta \delta \right\},
\]
where $C_\delta \triangleq \sqrt{\frac{12 C_L^2}{\rho}}$. Because $\Phi$ is proper, lower-semicontinuous, and bounded from below, the restriction of our iterates to the bounded sub-level set $\Omega \triangleq \{x \in \mathbb{R}^d \mid \mathcal{L}_k \le \mathcal{L}_0\}$ ensures that $\Omega$ is a compact domain. By the KL property, there exists a concave, continuously differentiable desingularizing function $\phi(s) = c \cdot s^{1-\theta}$ for an exponent $\theta \in [0, 1)$ such that for all active parameter iterations remaining outside the neighborhood profile ($x_k \in \Omega \setminus \mathcal{U}_\delta$), the local geometry satisfies the functional gradient bounding inequality
\[
\phi'(\Phi(x_k) - \Phi^*) \cdot \operatorname{dist}(0, \partial \Phi(x_k)) \ge 1.
\]
Equivalently, this relation states that the subdifferential magnitude outside our neighborhood is uniformly lower bounded by the inverse slope of the desingularizing metric
\[
\operatorname{dist}(0, \partial \Phi(x_k)) \ge \frac{1}{\phi'(\Phi(x_k) - \Phi^*)}.
\]
Since $\mathcal{L}_k$ is a non-negative quasi-supermartingale (up to the summable noise $\eta_k^2 \delta^2 + C_2 \eta_k^3$), application of the Robbins–Siegmund theorem guarantees that the sequence $\{x_k\}$ remains within the compact sub-level set $\Omega$ almost surely, and the functional sequence $\mathcal{L}_k$ converges to a finite random variable. Recall that the total kinetic energy of the optimizer is bounded, i.e., $\sum_{k=0}^\infty \mathbb{E}[\|x_{k+1} - x_k\|^2] < \infty$, which strictly dictates that $\lim_{k \to \infty} \mathbb{E}[\|x_{k+1} - x_k\|^2] = 0$. Because the distance function $\operatorname{dist}(\cdot, \mathcal{U}_\delta)$ is $1$-Lipschitz continuous, the step-to-step distance variation satisfies
\[
\left| \operatorname{dist}(x_{k+1}, \mathcal{U}_\delta) - \operatorname{dist}(x_k, \mathcal{U}_\delta) \right| \le \|x_{k+1} - x_k\|.
\]
Taking expectations reveals that the variation between consecutive distance profiles vanishes asymptotically ($\lim_{k \to \infty} \mathbb{E}[| \operatorname{dist}(x_{k+1}, \mathcal{U}_\delta) - \operatorname{dist}(x_k, \mathcal{U}_\delta) |] = 0$). Combining this vanishing variation property with our core subsequence tracking condition, $\liminf_{k \to \infty} \mathbb{E}\left[\operatorname{dist}(0, \partial \Phi(x_k))^2\right] \le C_\delta^2 \delta^2$, 
\[
\lim_{k \to \infty} \mathbb{E}\left[\operatorname{dist}(x_k, \mathcal{U}_\delta)\right] = 0.
\]
\end{proof}
\end{document}